\documentclass[12pt,letterpaper,reqno]{amsart}
\usepackage{amsmath}
\usepackage{amsfonts}
\usepackage{amssymb}
\usepackage{graphicx}
\usepackage{xcolor}
\usepackage{bm} 
\usepackage{hyperref}

\newtheorem{theorem}{Theorem}
\newtheorem{lemma}[theorem]{Lemma}
\newtheorem{proposition}[theorem]{Proposition}

\newtheorem{corollary}[theorem]{Corollary}

\newtheorem{problem}[theorem]{Problem}

\newcommand{\norm}[1]{\left\lVert #1 \right\rVert}

\allowdisplaybreaks

\newcommand{\ip}[2]{\langle #1,#2\rangle}

\theoremstyle{definition}
\newtheorem{remark}[theorem]{Remark}

\begin{document}
	\author{}

			\address{Y.~Huang, Department of Mathematical and Statistical Sciences, University of Alberta, Edmonton, AB T6G 2N8, Canada}
	
	\email{yhuang32@ualberta.ca}

	\address{S.~Myroshnychenko, Department of Mathematics and Statistics, University of the Fraser Valley, Abbotsford, BC V2S 7M8, Canada}
	
	\email{serhii.myroshnychenko@ufv.ca}

	\address{K.~Tatarko, Department of Pure Mathematics, University of Waterloo, Waterloo, ON N2L 3G1, Canada}
	
	\email{ktatarko@uwaterloo.ca}

	\address{V.~Yaskin, Department of Mathematical and Statistical Sciences, University of Alberta, Edmonton, AB T6G 2N8, Canada}
	
	\email{yaskin@ualberta.ca}

	\thanks{All authors were supported in part by NSERC}

	\subjclass[2020]{Primary 52A20, 52A40}
	
	\keywords{Convex body, section, centroid, Fourier transform, hyperbolic space, $s$-concave function}
	
\title[Analogues of questions of Gr\"unbaum and Loewner]{On hyperbolic and functional analogues of questions of Gr\"unbaum and Loewner}

\author[Y. Huang, S. Myroshnychenko, K. Tatarko, and V. Yaskin]{Yu Huang, Sergii Myroshnychenko, Kateryna Tatarko,\\ and Vladyslav Yaskin}

\begin{abstract} 
Myroshnychenko, Tatarko, and Yaskin constructed a body $K$  in $\mathbb{R}^n$, $n \geq 5$, with the property that there is exactly one hyperplane $H$ passing through $c(K)$, the centroid of $K$, such that the centroid of $K\cap H$ coincides with $c(K)$.
This construction provided answers to  questions of Gr\"unbaum and Loewner for $n\geq 5$, which are still open in dimensions $3$ and $4$. We study analogues of these questions in the settings of  hyperbolic space $\mathbb H^n$ and   $s$-concave functions  on $\mathbb R^n$.
\end{abstract}

\maketitle

\section{Introduction}

Let $K$ be a convex body in $\mathbb{R}^n$, i.e., a compact convex set with non-empty interior. The centroid of $K$ is the point
\begin{equation}\label{def-centroid}
c(K) = \frac{1}{|K|} \int_K x \, dx,
\end{equation}
where integration is with respect to Lebesgue measure and $|K|$ denotes the volume of $K$.

Consider the family of hyperplanes in $\mathbb R^n$ containing the centroid of $K$. If $K$ is centrally symmetric then $c(K\cap H)=c(K)$ for each hyperplane $H$ from this family. For general convex bodies this is no longer true. It is natural to ask  how many  sections with the property $c(K\cap H)=c(K)$ every  convex body $K$ in $\mathbb R^n$ has. The following problems were posed by Gr\"unbaum \cite{G1961} and Loewner \cite[Problem 28]{F}; see also \cite[A8]{CFG}.

\begin{problem}\label{GrunbaumProblem} (Gr\"unbaum)
	Is the centroid $c(K)$ of $K \subset \mathbb{R}^n$ the centroid of at least $n + 1$ different $(n-1)$-dimensional sections of $K$ through $c(K)$?
\end{problem}

\begin{problem}\label{LoewnerProblem} (Loewner)
	Let $\mu(K)$ be the number of hyperplane sections of $K$ passing through $c(K)$ whose centroid is the same as $c(K)$. Let $\mu(n) =\displaystyle \min_{K\in \mathcal K^n} \mu(K)$ where $\mathcal K^n$ is the class of all convex bodies in $\mathbb R^n$. What is the value of $\mu(n)$?

\end{problem}

It is easy to show that $\mu(2)=3$; as was noticed by Gr\"unbaum \cite{G1961} and Loewner \cite[Problem 28]{F}. If $n\ge 3$, Gr\"unbaum \cite{G1961} has shown that $\mu(n)\ge 1$;  see also \cite{MR}. 
Myroshnychenko, Tatarko, and Yaskin \cite{MTY2} have shown  that $\mu(n)=1$ for $n\ge 5$. The case of dimensions $n=3$ and $n=4$ is still open. It is natural to study the problem in other settings, where we can obtain the answer in all dimensions.
In this paper we show that, in  hyperbolic space $\mathbb H^n$ and in the case of $s$-concave functions on $\mathbb R^n$ with $-1/(n+1)<s<\infty$, the analogue of the number $\mu(n)$ equals 1 for all $n\ge 3$.

For other recent results about centroids of convex bodies the reader is referred to \cite{BHLL}, \cite{BHPS},  \cite{Fr}, \cite{FMY}, \cite{LY}, \cite{MNRY},  \cite{MSZ}, \cite{MTY1}, \cite{NRY}, \cite{PTW}, \cite{ShY},  \cite{SY}, \cite{SZ}.

\section{Preliminaries}

\subsection{Hyperbolic geometry}
We will start with the hyperboloid model of hyperbolic space $\mathbb H^n$. The reader is referred to \cite{R} for additional background information. Let $\mathbb R^{n,1}$ be the Minkowski space, which can be identified with $\mathbb R^{n+1}$ equipped with the Minkowski inner product
\[
\langle x, y\rangle_{n,1} =-x_0y_0+x_1y_1+\cdots+x_ny_n,
\]
where $x=(x_0,x_1, \ldots, x_n)$ and $y=(y_0,y_1, \ldots, y_n)$.

In the hyperboloid model, we define 
\[
\mathcal{H}^n=\{x\in\mathbb R^{n,1}:\ \langle x, x\rangle_{n,1}=-1,\ x_0>0\},
\]
endowed with the induced metric $g_{\mathcal{H}}$. The corresponding volume element we denote by $ d\mathrm{vol}_{\mathcal{H}}$.

For a  set $L\subset\mathcal{H}^n$ of positive volume, 
define its moment vector
\begin{equation}\label{eq:Z-def}
	Z(L)=\int_{L} x\, d\mathrm{vol}_{\mathcal{H}}(x) = (Z_0(L),\bm{Z}(L))
	 \in\mathbb R^{n,1}
\end{equation}
where $Z_0(L) \in \mathbb{R}$ and $\bm{Z}(L) \in \mathbb{R}^n$.
This point is not on $\mathcal H^n$, but if we normalize it properly, then we will get the center of mass of $L$. 
Let
\begin{equation}\label{eq:m-def}
	m(L)=\sqrt{-\ip{Z(L)}{Z(L)}_{n,1}}
	=\sqrt{Z_0(L)^2-|\bm{Z}(L)|^2}.
\end{equation}
Here and throughout the paper, for a vector $p \in \mathbb R^n$,   $|p|$   denotes its Euclidean norm.
Then   the centroid of $L$ is the point in $\mathcal{H}^n$ defined by
\begin{equation}\label{eq:C-on-H}
	C(L)=\frac{Z(L)}{m(L)}.
\end{equation}
This definition of centroid in $\mathcal{H}^n$ is analogous to the one in the spherical space (see \cite{Ga} for a study of centroids for discrete sets in constant curvature spaces).

Instead of the hyperboloid model, it will be more convenient to use the Poincar\'e model of $\mathbb H^n$ in the unit ball $\mathbb{B}^n$ in $\mathbb R^n$. The identification between the two models can be described geometrically as follows.  Consider the unit Euclidean ball $\mathbb{B}^n$ in the hyperplane $\{x_0=0\} \subset \mathbb R^{n,1}$. Given  a point $x$  in $\mathcal H^n$, consider the line segment connecting $x$ and $(-1,0,\ldots, 0)$, The line segment intersects $\mathbb{B}^n$ at some point $p$. 
The corresponding map $F:\mathbb{B}^n\to\mathcal{H}^n$, that sends $p$ to $x$, is given by
\begin{equation}\label{eq:Cayley}
	F(p)=\Big(\frac{1+|p|^2}{1-|p|^2},\,\frac{2p}{1-|p|^2}\Big),
	\qquad p\in\mathbb{B}^n.
\end{equation}

Using $F$ to pull back the  metric from $\mathcal H^n$, we obtain the metric in  the Poincar\'e model:
\begin{equation}\label{eq:Poincare-metric}
	g_{\mathbb{B}}=\frac{4}{(1-|p|^2)^2}\sum_{i=1}^n dp_i^2,
\end{equation}
and hence its volume element is
\begin{equation}\label{eq:volB}
	d\mathrm{vol}_{\mathbb{B}}
	= \frac{2^n}{(1-|p|^2)^n}\,dp,
\end{equation}
where $dp$ denotes  Lebesgue measure on $\mathbb R^n$.

Any two points in the Poincar\'e model can be connected by a unique geodesic segment. A set $K\subset \mathbb B^n$ is called {\it convex} if for any $p_1$ and $p_2$ in $K$, the geodesic segment connecting these points lies in $K$.    Let $K$ be a convex body in the Poincar\'e model, i.e., a compact convex set with non-empty interior. We define the  centroid   of $K$ as the  preimage of \eqref{eq:C-on-H} under the map $F$.

From \eqref{eq:Z-def}, using \eqref{eq:Cayley} and \eqref{eq:volB}, we obtain
\begin{align}
	Z_0(K) &= \int_{K}\frac{2^n\,(1+|p|^2)}{(1-|p|^2)^{n+1}}\,dp,\label{eq:Z0-unif}\\
	\bm{Z}(K) &= \int_{K}\frac{2^{n+1}\,p}{(1-|p|^2)^{n+1}}\,dp.\label{eq:Zvec-unif}
\end{align}
From  \eqref{eq:Cayley} we see  that if $x=F(p)$, then $p_i=x_i/(x_0+1)$, $i=1,\ldots, n$. Thus

\begin{equation}\label{eq:ball-centroid}
	C(K)=\frac{\bm{Z}(K)/m(K)}{Z_0(K)/m(K)+1}=\frac{\bm{Z}(K) }{Z_0(K)+m(K)},
\end{equation}
where $m(K)=\sqrt{Z_0(K)^2-|\bm{Z}(K)|^2}$.


For any vector $\xi$  on the unit sphere $S^{n-1}$, consider the hypersurface $\xi_{\mathbb B}^\perp$ in $\mathbb B^n$ passing through the origin defined by 
$$\xi_{\mathbb B}^\perp = \{x\in \mathbb B^n: \xi_1 x_1+ \cdots+\xi_n x_n=0\}.$$
Such a hypersurface is a totally geodesic submanifold in the Poincar\'e ball model of hyperbolic space (they are analogous to hyperplanes in the Euclidean space in the sense that every geodesic line on such a surface is also a geodesic in the ambient space). 

Let $K$ be a convex body in $\mathbb B^n$ that contains the origin in its interior. For any $\xi\in S^{n-1}$, the section of $K$ by the hypersurface $\xi_{\mathbb B}^\perp$ is a convex body in $\xi_{\mathbb B}^\perp$ and its centroid can be expressed analogously to \eqref{eq:ball-centroid}, with $n$ replaced by $n-1$:
\begin{equation}\label{eq:section-centroid}
	C(K\cap \xi_{\mathbb B}^\perp)= \frac{1}{Z_0(K\cap\xi_{\mathbb B}^\perp)+m(K\cap \xi_{\mathbb B}^\perp)} \int_{K\cap \xi_{\mathbb B}^\perp}\frac{2^{n}\,p}{(1-|p|^2)^{n}}\,dp,
\end{equation}

Analogous to Problem 2, we can consider the following question in hyperbolic space.
\begin{problem}\label{HyperbolicProblem}
	Let $K$ be a convex body in $\mathbb H^n$ and  $\eta(K)$ be the number of totally geodesic $(n-1)$-dimensional sections of $K$ passing through $c(K)$ whose centroid is the same as $c(K)$. Let $\eta(n) =\displaystyle \min_{K\in \mathcal K({\mathbb H}^n)} \eta(K)$ where $\mathcal K({\mathbb H}^n)$ is the class of all convex bodies in $\mathbb H^n$. What is the value of $\eta(n)$?
	\end{problem}

	Since the body $K\subset \mathbb B^n$ can be identified with a body in $\mathbb R^n$, we can apply to $K$ standard Euclidean concepts. 
	We say that a compact set $K \subset \mathbb{R}^n$ is star-shaped about the origin $0$ if for every point $x \in K$ each point of the interval $[0, x)$ is an interior point of $K$. The {\it Minkowski functional} of $K$ is defined by
	$$
	\|x\|_K = \min\{\lambda \geq 0:  x \in \lambda K \}.
	$$
	We say that $K$ is a {\it  star body} if it is compact, star-shaped about the origin and its Minkowski functional is a continuous function on $\mathbb{R}^n$.
	
	The {\it radial function} of a star body $K$ is defined by
	$$
	\rho_K(\xi) = \max \{\lambda > 0: \lambda \xi \in K \}, \quad \xi \in S^{n-1}.
	$$
	Observe that $\rho_K(\xi) = \|\xi\|_K^{-1}$ for any $\xi \in S^{n-1}$ and $\rho_K$ is positive and continuous on $S^{n-1}$.

	We say that $K$ is origin symmetric if $x\in K \Leftrightarrow -x \in K$. For an origin symmetric star body $K$, its radial function $\rho_K$ is an even function on the sphere, i.e.,  $\rho_K(\xi)=\rho_K(-\xi) $ for all $\xi\in S^{n-1}$.

\subsection{$s$-concave functions}	
	
Let us now discuss functional versions of Problem 2. Let $-\infty \le s\le \infty$.   A function $f: \mathbb R^n\to \mathbb [0,\infty)$ is called $s$-concave  if  
\begin{equation}\label{s-concave}
	f(\lambda x + (1-\lambda)y)\ge \left(\lambda f^s(x)+ (1-\lambda)f^s(y)\right)^{1/s},
\end{equation}
for all $x, y\in \mathbb R^n$ such that $f(x)\cdot f(y)>0$ and all $\lambda\in (0,1)$.

If $s =-\infty$, $0$, $\infty$, the definition above is understood in the sense of limits. In particular, 
$f$ is $\infty$-concave if 
$$f(\lambda x + (1-\lambda)y)\ge\max\{f(x),f(y)\},$$
for all $x, y\in \mathbb R^n$ such that $f(x)\cdot f(y)>0$ and all $\lambda\in (0,1)$. Such functions are constant multiples of indicator functions of convex sets. 

If $s=0$,  inequality \eqref{s-concave}  becomes 
$$f(\lambda x + (1-\lambda)y)\ge   f^\lambda(x) f^{1-\lambda}(y),$$
for all $x, y\in \mathbb R^n$   and all $\lambda\in (0,1)$. Such functions are called log-concave.

We will denote by $C_s(\mathbb R^n)$ the class of $s$-concave functions  on $\mathbb R^n$ with positive finite integrals. It is known that these classes become larger when $s$ gets smaller.

Below we will focus on  functions from $C_s(\mathbb R^n)$ with $-1/(n+1)<s<\infty$. If $f\in C_s(\mathbb R^n)$ with $s\ge 0$, then $$f(x)\le Ae^{-B|x|},$$ for all $x\in \mathbb R^n$ and some positive constants $A$ and $B$; see \cite[Lemma 2.2.1]{BGVV}. If $-1/(n+1)<s<0$,   then there is a constant $C>0$, such that 
\begin{equation}\label{s-conc}
	f(x)\le \frac{C}{1+|x|^{-1/s}},
\end{equation}
for all $x\in \mathbb R^n$; see \cite{B}. Thus, if $f\in C_s(\mathbb R^n)$ with $-1/(n+1)<s<\infty$, then its first moments exist, and we can define
 its centroid (or barycenter)   analogously to  \eqref{def-centroid}:
$$c(f) = \frac{\int_{\mathbb R^n} x f(x)\, dx}{\int_{\mathbb R^n} f(x)\, dx} .$$
If $H$ is an affine subspace of $\mathbb R^n$, then we denote
by 
$$c_H(f)=\frac{\int_{H} x f(x)\, dx}{\int_{ H} f(x)\, dx} $$ the centroid of the restriction of $f$ to the subspace $H$.

We will now formulate an analogue of Problem \ref{LoewnerProblem} for $s$-concave functions.

\begin{problem}\label{FunctionalProblem}
	Let $f \in C_s({\mathbb R}^n)$ with $-1/(n+1) <s<\infty$, and let $\nu(f)$ be the number of   $(n-1)$-dimensional affine subspaces $H$ of $\mathbb R^n$ passing through $c(f)$ such that $c_H(f)=c(f)$.  What is the value of $\nu_s(n) =\displaystyle \min_{f\in \mathcal C_s({\mathbb R}^n)} \nu(f)$?  
\end{problem}

\section{Main results}
\subsection{The hyperbolic case}\label{Section:H}

We will first show that $\eta (2)=3$. The proof is similar to that for the Euclidean case; see e.g., \cite[Section~2]{MTY2}.
To show that $\eta(2)\le 3$, consider an equilateral Euclidean triangle $\Delta$ in $\mathbb B^2$ with centroid at the origin. Observe that $\Delta$ is also convex in the hyperbolic sense, and the hyperbolic centroid of $\Delta$ is  at the origin, since the integral
$$
\int_{\Delta}\frac{p}{(1-|p|^2)^{3}}\,dp
$$ is invariant under rotations by $2\pi/3$ with respect to the origin.

Note that a chord whose hyperbolic length is bisected by the origin is precisely a chord whose Euclidean length is bisected by the origin. Since, in the Euclidean setting, the origin bisects exactly three chords of $\Delta$, the same holds in the hyperbolic setting.

Now consider an arbitrary convex body $K\subset \mathbb B^2$ with centroid at the origin, i.e.,
$$
\int_{K}\frac{p}{(1-|p|^2)^{3}}\,dp=0.
$$
Passing to polar coordinates, we obtain
\begin{equation*}
	\int_0^{2\pi} \Psi(\rho(\varphi)) \cos(\varphi) d\varphi = 0 \quad \text{and} \quad  \int_0^{2\pi}  \Psi(\rho(\varphi)) \sin(\varphi) d\varphi = 0,
\end{equation*}
where $$\Psi (s) = \int_0^s\frac{r^2}{(1-r^2)^3}\, dr,$$
and $\rho(\varphi)$ is the radial function of $K$   in polar coordinates.

Using the same argument as in \cite[Section~2]{MTY2}, one can show that the function $$ \Psi(\rho(\varphi)) - \Psi(\rho(\varphi+\pi))$$
has at least three roots in the interval $[0,\pi)$, which means that $\rho(\varphi) = \rho(\varphi+\pi)$ for at least three values of $\varphi\in [0,\pi)$. Thus, at least three chords of $K$ are bisected by the origin, i.e., $\eta(2)\ge 3$. Recalling that $\eta(\Delta)= 3$, we obtain $\eta(2)=3$.

We will now present the main result of this section. The proof is based on the Fourier transform of distributions. The reader is referred to \cite{GS} and \cite{K} for background information.

\begin{theorem}\label{Main:H}
	There exists a convex body $K \subset \mathbb{B}^n$,  $n \geq 3$, with  centroid at the origin, such that 
		$$
 C(K \cap \xi^{\perp})\in\{x\in \mathbb B^n: x_n>0\}
	$$
for all $\xi \ne \pm e_n$.
\end{theorem}
\begin{proof}

Our goal is to construct a body $K$, with centroid at the origin such that
\begin{align}
	\bm{Z}_n(K\cap\xi_{\mathbb{B}}^\perp)
	&= \int_{K\cap\xi_{\mathbb{B}}^\perp}\frac{2^{n}\,p_n}{(1-|p|^2)^{n}}\,dp\label{eq:Zvecsec-unif}
\end{align}
is positive for all $\xi \ne \pm e_n$.

As was shown in \cite{Y1} (see the proof of Proposition 3.9), there exists 
   an origin-symmetric convex body  $M\subset \mathbb B^n$ with strictly positive principal curvatures and $C^\infty$ boundary such that 
 $$	\Phi(x)=\frac{\norm{x}_M^{-1}}{1-\bigl(|x|/\norm{x}_M\bigr)^2}$$
 is not a positive definite distribution on $\mathbb R^n$.	
 
 Moreover, we can assume that $M$ is rotationally invariant about the $x_n$-axis, and $\widehat\Phi (e_n)<0$. Since $\Phi$ is a homogeneous function of degree $-1$ that is infinitely smooth on $\mathbb R^n \setminus\{0\}$, $\widehat\Phi$ is a homogeneous function of degree $-n+1$ that is also infinitely smooth on $\mathbb R^n \setminus\{0\}$; see \cite[Corollary 3.17]{K}.

 Let $\Omega(e_n)$ and $\Omega(-e_n) \subset S^{n-1}$ be  open spherical balls centered at $e_n$ and $-e_n$ respectively such that $\widehat\Phi(\xi) < 0$ for all $\xi \in \Omega(\pm e_n).$   Define an even function $G \in C^{\infty}(S^{n-1})$ that is invariant under rotations about the $x_n$-axis and such that
	$$
	G(\xi)  = \begin{cases}
		\text{positive}, \quad \xi \in \Omega(\pm e_n) \backslash \{\pm e_n\};\\
		0, \quad \xi \in \{\pm e_n\} \cup S^{n-1} \backslash \Omega(\pm e_n).
	\end{cases}
	$$
	By construction,
\begin{equation}\label{int_G}
	\int_{S^{n-1}}\widehat\Phi(\xi) G(\xi) d\xi < 0.
\end{equation}
	Next, we define the function $H \in C^{\infty}(S^{n-1})$ as 
	$$
	H(x) = |x|^{-1} - (4(x_1^2 + \dots + x_{n-1}^2) + x_n^2)^{-\frac12}.
	$$
	Note that $H(x) > 0$ if $x\in S^{n-1}\setminus\{\pm e_n\}$ and $H(\pm e_n) = 0$.  Extending $H$ to $\mathbb R^n\setminus\{ 0\}$ as a homogeneous function of degree $-1$ and computing its Fourier transform using the well-known formulas 
	
	$$\left(|\cdot|^{-1} \right)^\wedge (x) = c_n |x|^{-n+1},$$ 	and 
	$$\left(|T y|^{-1} \right)^\wedge (x) = c_n |\det T|^{-1} |T^{-t}x|^{-n+1},$$
	where 	  $c_{n} = \frac{2^{n-1} \pi^{\frac{n}2} \Gamma\left(\frac{n-1}2\right)}{\Gamma\left(\frac12\right)}$ and $T$ is an invertible linear transformation on $\mathbb R^n$, we obtain	
	\begin{align*}
		\widehat{H}(x) 
		&=c_n \left(|x|^{-n+1} - \left(x_1^2 + \dots + x_{n-1}^2 + 4x_n^2\right)^{\frac{-n+1}2}\right).
	\end{align*}
Since $\widehat{H} (x) \ge 0$ for $x\in \mathbb R^n\setminus\{0\}$, an application of the spherical Parseval formula (see \cite[Lemma 3.22]{K}) gives
	\begin{equation}\label{int_H}
		\int_{S^{n-1}} \widehat\Phi(\xi) H(\xi) d\xi =  \int_{S^{n-1}} \Phi(\xi) \widehat{H}(\xi) d\xi > 0.
		\end{equation}
	Now, for $\lambda\in [0,1]$,   we define 
	$$
	g_\lambda (\xi) = (1-\lambda)G (\xi) + \lambda H (\xi), \qquad  \xi\in S^{n-1}.
	$$
	Observe that $g_\lambda\in C^\infty(S^{n-1})$,  $g_{\lambda}(\xi) > 0$ for all $\xi \ne \pm e_n$, and $g_\lambda(\pm e_n) = 0$.
 
  Extending $g_\lambda$ to  $ \mathbb R^n\setminus\{0\}$ as a homogeneous function of degree $-1$ and denoting  the Fourier transform of this extension by $\widehat{g_\lambda}$, we define a function $\phi_\lambda$ on $S^{n-1}$ by the formula
	$$\phi_\lambda(\xi) = \frac{1}{\xi_n} \widehat{g_\lambda}(\xi), \qquad \xi\in S^{n-1}.$$
	Note that $\widehat{g_\lambda}(\xi)=0$ when $\xi\in e_n^\perp$, and   defining $\phi_\lambda$ to be zero on $e_n^\perp$ makes it a $C^\infty$ function on $S^{n-1}$ (for details see \cite[p.~8]{MTY2}).  Also observe that $\phi_\lambda$ is odd.

Consider the strictly increasing function
	\begin{equation}\label{eq:def-Psi}
		\Psi(s) := \int_0^s \frac{r^{n-1}}{(1-r^2)^n}\,dr,\qquad s\in[0,1).
	\end{equation}
	For $\lambda\in[0,1]$ and  small enough $\varepsilon>0$, we define the radial
	function $\rho_{K_{\lambda,\varepsilon}}$  of a star body $K_{\lambda,\varepsilon}$ by
	\begin{equation}\label{eq:perturb}
		\Psi(\rho_{K_{\lambda,\varepsilon}}(\theta))
		= \Psi(\rho_M(\theta))+\varepsilon\,\phi_\lambda(\theta),
		\qquad \theta\in S^{n-1}.
	\end{equation}	
Since $M$ is convex with strictly positive principal curvatures, there is $\varepsilon_\ast>0$ such that	  $K_{\lambda,\varepsilon}$ is convex for all $0<\varepsilon <\varepsilon_\ast$ and all $\lambda\in[0,1]$.  Also, $\rho_{K_{\lambda,\varepsilon}}$ is rotationally invariant about the $x_n$-axis since $\rho_M$ and $\phi_\lambda$ are rotationally invariant and $\Psi$ is strictly increasing.

For every $\xi\in S^{n-1}$,	we   have 	
	\begin{align*}
		\bm{Z}_n(K_{\lambda,\varepsilon}\cap\xi_{\mathbb{B}}^\perp)
		&= \int_{K_{\lambda,\varepsilon}\cap\xi_{\mathbb{B}}^\perp}\frac{2^{n}\,p_n}{(1-|p|^2)^{n}}\,dp \\
		&= \int_{S^{n-1}\cap\xi^\perp}\int_0^{\rho_{K_{\lambda,\varepsilon}}(\theta)} r^{n-2} \frac{2^{n}r \, \theta_n}{(1-r^2)^{n}}\,dr\, d\theta\\
		&= 2^n\int_{S^{n-1}\cap\xi^\perp}  \Psi(\rho_{K_{\lambda,\varepsilon}}(\theta))  \, \theta_n \, d\theta\\
		&= 2^n\int_{S^{n-1}\cap\xi^\perp}  \left( \Psi(\rho_M(\theta))+\varepsilon\,\phi_\lambda(\theta)\right)   \theta_n \, d\theta\\
		&= 2^n \varepsilon\int_{S^{n-1}\cap\xi^\perp}  \phi_\lambda(\theta)   \, \theta_n \, d\theta\\
		& = \frac{2^n \varepsilon}{\pi} \left(\phi_\lambda(x) x_n\right)^\wedge (\xi)  = \frac{2^n \varepsilon}{\pi} \left(\widehat g_\lambda \right)^\wedge (\xi) = 4^n\pi^{n-1}\varepsilon g_{\lambda}(\xi)\ge 0.
	\end{align*}
	Above we used the fact that if $f$ is an even continuous function of degree $-n+1$ on $\mathbb R^n\setminus\{0\}$, then its Fourier transform is a homogeneous function of degree $-1$ whose restriction to $S^{n-1}$ equals $$\widehat f (\xi) = \pi \int_{S^{n-1}\cap \xi^\perp} f(\theta)\, d\theta,\qquad \xi\in S^{n-1},$$
	see \cite[Lemma 3.7]{K}.

	We will now choose $\lambda$ and $\varepsilon$ so that the centroid of $K_{\lambda,\varepsilon}$ is at the origin. 
First, we note that
\begin{align*}
	\bm{Z}_i(K_{\lambda,\varepsilon})  = \int_{K_{\lambda,\varepsilon}}\frac{2^{n+1}\,p_i}{(1-|p|^2)^{n+1}}\,dp = 2^{n+1} \int_{S^{n-1}}\int_0^{\rho_{K_{\lambda,\varepsilon}}(\theta)}   \frac{r^n\, \theta_i}{(1-r^2)^{n+1}}\,dr\, d\theta =0
\end{align*}
for $i = 1, \dots, n-1$, where we used that $\rho_{K_{\lambda,\varepsilon}}$ is rotationally invariant about the $x_n$-axis. Therefore, it remains to show that $\bm{Z}_n(K_{\lambda,\varepsilon}) = 0$. We have	
\begin{align*}	\bm{Z}_n(K_{\lambda,\varepsilon})  = \int_{K_{\lambda,\varepsilon}}\frac{2^{n+1}\,p_n}{(1-|p|^2)^{n+1}}\,dp  = 2^{n+1} \int_{S^{n-1}}\int_0^{\rho_{K_{\lambda,\varepsilon}}(\theta)}   \frac{r^n\, \theta_n}{(1-r^2)^{n+1}}\,dr\, d\theta.
	\end{align*}	
Considering the latter as a function of  $\varepsilon$, we will obtain  its expansion for small $\varepsilon>0$.
	Using  \eqref{eq:perturb} with $\varepsilon =0$, we have  
	$$	\Psi(\rho_{K_{\lambda,\varepsilon}}(\theta))\left.\right|_{\varepsilon=0}
	= \Psi(\rho_M(\theta)),   \mbox{ i.e., }   \rho_{K_{\lambda,\varepsilon}}(\theta)\left.\right|_{\varepsilon=0}=\rho_M(\theta), \mbox{ for all } \theta\in S^{n-1}.$$
	Additionally, differentiating \eqref{eq:perturb}, we get
$$\left.\frac{d}{d\varepsilon}	\rho_{K_{\lambda,\varepsilon}}(\theta) \right|_{\varepsilon=0}=\frac{\phi_{\lambda}(\theta)}{\Psi'(\rho_M(\theta))}=\frac{\phi_{\lambda}(\theta)(1-\rho_M^2(\theta))^n}{ \rho_M^{n-1}(\theta)}.$$
Thus,
\begin{align*} \int_0^{\rho_{K_{\lambda,\varepsilon}}(\theta)}&  \frac{ r^n  }{(1-r^2)^{n+1}}\,dr \\
	& = \int_0^{\rho_M(\theta)}   \frac{ r^n  }{(1-r^2)^{n+1}}\,dr + \varepsilon     \frac{ \rho_M^n(\theta) }{(1-\rho_M^2(\theta))^{n+1}}
\frac{\phi_{\lambda}(\theta)(1-\rho_M^2(\theta))^n}{ \rho_M^{n-1}(\theta)} +\varepsilon^2 R_{\lambda,\varepsilon} (\theta)\\
&= \int_0^{\rho_M(\theta)}   \frac{ r^n  }{(1-r^2)^{n+1}}\,dr + \varepsilon     \frac{ \rho_M(\theta) }{1-\rho_M^2(\theta)}
 \phi_{\lambda}(\theta) +\varepsilon^2 R_{\lambda,\varepsilon} (\theta),
 \end{align*}
 where the last term is the remainder in the Taylor expansion.

Since $M$ is origin-symmetric, its radial function $\rho_M$ is an even function on the sphere, and thus
\begin{align*}  \int_{S^{n-1}}\int_0^{\rho_M(\theta)}   \frac{r^n\, \theta_n}{(1-r^2)^{n+1}}\,dr\, d\theta=0.
\end{align*}
Therefore,
\begin{align*}	\bm{Z}_n({K_{\lambda,\varepsilon}}) 
	& =  2^{n+1}\varepsilon \int_{S^{n-1}}  \Phi(\theta) \theta_n\phi_{\lambda}(\theta) \, d\theta + \varepsilon^2 \bar R_{\lambda,\varepsilon}\\
		& =  2^{n+1}\varepsilon \int_{S^{n-1}}  \Phi(\theta) \widehat g_{\lambda}(\theta) \, d\theta + \varepsilon^2 \bar R_{\lambda,\varepsilon}\\
				& =  2^{n+1}\varepsilon \int_{S^{n-1}}  \widehat \Phi(\theta)  g_{\lambda}(\theta) \, d\theta + \varepsilon^2 \bar R_{\lambda,\varepsilon},
\end{align*}
where
$$  \bar R_{\lambda,\varepsilon} = 2^{n+1}\int_{S^{n-1}} R_{\lambda,\varepsilon}(\theta)\, \theta_n \, d\theta.$$
Note that $\bar R_{\lambda,\varepsilon}$ is a continuous function of $\varepsilon$ and $\lambda$. 

	Define
	$$F(\lambda,\varepsilon)=\int_{S^{n-1}} \widehat\Phi(\xi) g_\lambda (\xi) \, d\xi+\varepsilon 2^{-n-1}  \bar R_{\lambda,\varepsilon}.$$
	
Using \eqref{int_G} and \eqref{int_H}, we obtain
	$$F(0,0)=\int_{S^{n-1}} \widehat\Phi(\xi)    G(\xi) \, d\xi<0$$
	and 
	$$F(1,0)=\int_{S^{n-1}}  \widehat\Phi(\xi) H(\xi) \, d\xi>0.$$

	Since $F(\lambda, \varepsilon)$ is a continuous map on $[0,1] \times [0, \varepsilon_\ast]$, there exists a small $\varepsilon_0>0$ and $\lambda_0 \in [0,1]$ such that $F(\lambda_0, \varepsilon_0) = 0$, implying that  the centroid of $K_{\lambda_0,\varepsilon_0}$ is at the origin.
	\end{proof}
	
	Theorem 	\ref{Main:H} implies that	$\eta (n)\le 1$ for $n\ge 3$. As we will see later in Remark \ref{rem-function}, every convex body $K$ in $\mathbb H^n$, $n\ge 3$, has at least one hyperplane section $H$ passing through the centroid of $K$ such that $c(K\cap H) = c(K)$. Thus we obtain the following corollary. 
	
\begin{corollary}\label{H:n_ge_3}
	$\eta (n)=1$ for $n\ge 3$.
		\end{corollary}

 \begin{remark}
 	Using a similar construction  for the sphere $S^{n}\subset \mathbb R^{n+1}$, one can show that the spherical analogue of the number $\eta(n)$ is equal to 1 for all $n\ge 5$. 
 	 However, as in the Euclidean space, our construction does not work in dimensions $n=3$ and $n=4$.  Thus, we omit the details. 
 	\end{remark}

			\subsection{The functional case}

			The following lemma is a functional analogue of the remark on p.~352 in \cite{MR}. One can also use topological methods as in  \cite[Lemma~8]{MR}, but we prefer to give a simple proof using analysis.
			\begin{lemma}\label{Lemma:MR}
				Let  $f\in C_s(\mathbb R^n)$ with $-1/(n+1)<s<\infty$  and $0 \in \text{int}\left(\text{supp}(f)\right)$. Then there exists at least one direction $u \in S^{n-1}$ such that 
				$$
				\int_{u^\perp} x f(x) \, dx = 0.
				$$
			\end{lemma}
			\begin{proof} Since an $s_1$-concave function is also $s_2$-concave for all $s_2<s_1$, we can assume that $f$ is   $s$-concave with $-1/(n+1)<s<0$.
				As was mentioned in the introduction, the integrals $$
				\int_{u^\perp} x f(x) \, dx 
				$$
				are well defined for all $u\in S^{n-1}$.
				
				Consider the following function of $u \in S^{n-1}$,
				$$
				F(u) = \int_{\{x: \, \langle x, u\rangle \geq 0\} } f(x) \, dx.
				$$
				$F$ is continuous on $S^{n-1}$ since $f$ is integrable, and thus $F$ attains its extreme values. 
				
			Fix $u\in S^{n-1}$ and  let $v$ be a unit vector orthogonal to $u$. For a real number  $\varphi$ close to zero, define $$u_v(\varphi) = \cos\varphi\, u +\sin \varphi\, v.$$ We claim that 
				\begin{equation}\label{formula:der}
					\frac{d}{d\varphi} F(u_v(\varphi)) \Big|_{\varphi=0} =  \int_{u^\perp} \langle y,v\rangle f(y) \, dy.
					\end{equation}
				Indeed,
				\begin{align}\label{eqn:deriv}\frac{d}{d\varphi} F(u_v(\varphi)) \Big|_{\varphi=0} &=\lim_{\varphi\to 0} \frac1{\varphi} \left(F(u_v(\varphi))- F(u)\right)		\notag	\\
					&=\lim_{\varphi\to 0} \frac1{\varphi} \left( \int_{H^+} f(x) \, dx - \int_{H^-} f(x) \, dx\right)	,		
					\end{align}
						where $$H^+ = \{ x\in \mathbb R^n: \langle  x , u\rangle \leq 0 \ \text{and}\ \langle  x , u_v(\varphi)\rangle \geq 0 \}$$ and $$H^- = \{ x\in \mathbb R^n: \langle  x , u\rangle \geq 0 \ \text{and}\ \langle  x , u_v(\varphi)\rangle \leq 0 \}.$$
						It is enough to compute the limit in \eqref{eqn:deriv} as $\varphi\to 0^+$.  The case $\varphi\to 0^-$ will give the same result, since we can just replace $\varphi$ with $-\varphi$ and $v$ with $-v$.

						First, we will evaluate   $\lim_{\varphi \to 0^+}\frac{1}{\varphi}\int_{H^+} f(x) \, dx$.
							Note that any point $x \in H^+$ can be represented as $x = y + t u$, where $y \in u^\perp$, $ \langle y,v\rangle \ge 0$, and  $t \in [-\tan\varphi \langle y,v\rangle, 0]$.  Then
					\begin{align*}
					\lim_{\varphi\to 0^+} \frac1{\varphi}	\int_{H^+} f(x) \, dx & = \lim_{\varphi\to 0^+} \frac1{\varphi}\int\limits_{\{ y\in u^\perp:\langle y,v\rangle\ge 0\}} \left( \int\limits_{-\tan\varphi \langle y,v\rangle }^0 f(y + tu)\, dt \right) \,dy\\
					& = \lim_{\varphi\to 0^+} \frac1{\varphi}\int\limits_{\{ y\in u^\perp:\langle y,v\rangle\ge 0\}} \left(\int\limits_{-1 }^0 \tan\varphi \langle y,v\rangle f\left(y + \tan\varphi \langle y,v\rangle tu\right)\, dt \right)\, dy\\
					& = \int\limits_{\{ y\in u^\perp:\langle y,v\rangle\ge 0\}} \left( \int\limits_{-1 }^0 \lim_{\varphi\to 0^+}\left( \frac{\tan\varphi}{\varphi}  \langle y,v\rangle f\left(y + \tan\varphi \langle y,v\rangle tu\right)\right)\, dt \right)\, dy\\
					& = \int\limits_{\{ y\in u^\perp:\langle y,v\rangle\ge 0\}} \left(\int\limits_{-1 }^0   \langle y,v\rangle f\left(y \right) \, dt \right)\, dy\\
				& = \int\limits_{\{ y\in u^\perp:\langle y,v\rangle\ge 0\}}    \langle y,v\rangle f\left(y \right) \,   dy.
						\end{align*}						
							Above we used  the Dominated Convergence Theorem  to move the limit inside the integrals, since	by \eqref{s-conc} for small $\varphi$ 	we have					
				\begin{align*}
						\left| \frac{\tan\varphi}{\varphi}  \langle y,v\rangle f\left(y + \tan\varphi \langle y,v\rangle tu\right) \right| & \leq  2 |\langle y,v\rangle| f\left(y + \tan\varphi \langle y,v\rangle tu\right)\\
						& \leq  2 C |\langle y,v\rangle| {(1+|y + \tan\varphi \langle y,v\rangle tu|^{-1/s})^{-1}}\\
						& \le 2 C |\langle y,v\rangle| {(1+|y |^{-1/s})^{-1}},
					\end{align*}
						and the  latter is an integrable function on $u^\perp$.

						The assumption that the origin lies in the interior of the support of 
						$f$  was used in computing the limit 	$$\lim_{\varphi\to 0^+}f\left(y + \tan\varphi \langle y,v\rangle tu\right)= f\left(y\right).$$ Since $f$ is continuous in the interior of its support, which is a convex set,   the equality above holds for almost every  $y\in u^\perp$, except possibly on a set of  zero $(n-1)$-dimensional measure.

						Thus, we obtain
						$$\lim_{\varphi\to 0^+} \frac1{\varphi}	\int_{H^+} f(x) \, dx = \int\limits_{\{ y\in u^\perp:\langle y,v\rangle\ge 0\}} \langle y,v\rangle f(y)\, dy.$$
						Similarly, one can show 
							$$\lim_{\varphi\to 0^+} \frac1{\varphi}	\int_{H^-} f(x) \, dx =  - \int\limits_{\{ y\in u^\perp:\langle y,v\rangle\le 0\}} \langle y,v\rangle f(y)\, dy.$$
						Formula \eqref{formula:der} follows by subtracting the last two limits, and recalling that the limit as $\varphi\to 0^-$ gives the same answer.
							
					 Let $u_0 \in S^{n-1}$ be a direction, where $F$ attains its minimum. Then,   formula \eqref{formula:der} implies
					 $$\int_{u_0^\perp} \langle y,v\rangle f(y) \, dy=0,$$ for any $v\in u_0^\perp$, which yields the statement of the lemma.

			\end{proof}
			
					\begin{remark}\label{rem-function}
						Lemma~\ref{Lemma:MR} remains valid if $f$ is a continuous function supported on a star body. In particular if $$f(x) = \frac{1}{(1-|x|^2)^{n}} $$ supported on a star body $K$ that lies in the interior of $\mathbb B^n$, then there is $u\in S^{n-1}$ such that
						$$\int_{K\cap u^\perp}\frac{x}{(1-|x|^2)^{n}}\,dx=0.$$ This is precisely what we needed for the proof of Corollary \ref{H:n_ge_3}.
						
						\end{remark}
			
			\begin{proposition}\label{Main:logconcave}
				There exists   $f\in C_s(\mathbb R^n)$  with $-1/(n+1)<s<\infty$ and  centroid at the origin such that 
					$$
				\int_{u^\perp} x_n f(x) \, dx >0
				$$
				for all $u\in S^{n-1}$ other than $\pm e_n$.
				\end{proposition}
				
				\begin{proof}
					Let $g_s$ be an $s$-concave function defined as follows:
					\begin{enumerate} \item If $-1/(n+1)<s<0$, $$g_s(x) = (1+|x|^2)^{1/s};$$ \item if $s=0$,
					 $$g_s(x)= e^{-|x|^2};$$
					 \item if $s>0$, $$g_s(x) =\begin{cases}
					(9-|x|^2)^{1/s}	, & |x|<3,\\
					0, & |x|\ge 3.						
						\end{cases}$$
						\end{enumerate}

					Define $$f(x) = g_s(x) + \varepsilon x_n b(|x|),$$
					where $\varepsilon>0$ is sufficiently small and $b$ is an infinitely smooth function on $\mathbb R$ with support in the interval $[1,2]$ and such that
					  $$\int_1^2 r^{n+1} b(r)\, dr = 0$$ and $$\int_1^2 r^{n} b(r)\, dr > 0.$$

					  Since $b$ has compact support, which is in the interior of the support of $g_s$,   and $g_s$ is strictly $s$-concave, $f$ is $s$-concave for small enough $\varepsilon$.
					  
					  Let us show that the centroid of $f$ is at the origin. For each $i=1,\ldots, n$, the function $x_i g_s(x)$ is odd, and therefore
					  \begin{align*}\int_{\mathbb R^n} x_i f(x)\, dx&= \int_{\mathbb R^n} x_i \left(  g_s(x) + \varepsilon x_n b(|x|) \right)\, dx= \varepsilon \int_{\mathbb R^n} x_i    x_n b(|x|) \, dx\\
					  	&= \varepsilon \int_{S^{n-1} } \theta_i    \theta_n\, d\theta \int_1^2 r^{n+1} b(r) \, dr=0.					  	
					  	\end{align*}

					  	For the restriction of $f$ to the subspace $u^\perp$, we have		  	
					  		  \begin{align*}\int_{u^\perp} x_n f(x)\, dx&= \int_{u^\perp} x_n \left(  g_s(x) + \varepsilon x_n b(|x|) \right)\, dx= \varepsilon \int_{u^\perp}     x_n^2 b(|x|) \, dx\\
					  		&= \varepsilon \int_{S^{n-1}\cap u^\perp }    \theta_n^2 \, d\theta \int_1^2 r^{n} b(r) \, dr>0,					  		
					  	\end{align*}
					  	if $u\ne \pm e_n$. 
					  
					\end{proof}
		Using 	Lemma \ref{Lemma:MR} and Proposition	\ref{Main:logconcave} we obtain the following.
				
	\begin{corollary}
	Let $-1/(n+1)<s<\infty$. Then	$\nu_s (n)=1$ for $n\ge 2$.
	\end{corollary}

\bigskip

{\bf Acknowledgment.} We would like to thank Roman Vershynin for suggesting to study functional analogues of Problems~\ref{GrunbaumProblem} and \ref{LoewnerProblem}.

\end{document}